\documentclass[12pt]{amsart}

\usepackage[T1]{fontenc}

\usepackage{graphicx}
\usepackage{relsize}
\usepackage{t1enc}
\usepackage{latexsym}
\usepackage{amssymb}
\usepackage{graphicx}
\usepackage{amsmath}
\usepackage{amsthm}
\usepackage{amsfonts}
\usepackage{mathrsfs}
\usepackage[all]{xy}
\usepackage[british]{babel}
\usepackage[usenames]{xcolor}
\usepackage[hypertexnames=false,
    pdftex,
	pdfpagemode=UseNone,
	breaklinks=true,
	extension=pdf,
	colorlinks=true,
	linkcolor=blue,
	citecolor=blue,
	urlcolor=blue,
]{hyperref}

\numberwithin{equation}{section}

\theoremstyle{plain}
\newtheorem*{theorem*}{Main Theorem}					
\newtheorem{theorem}{Theorem}[section]	
\newtheorem{proposition}[theorem]{Proposition}		
\newtheorem{corollary}[theorem]{Corollary}
\newtheorem{lemma}[theorem]{Lemma}

\theoremstyle{definition}

\newtheorem{definition}[theorem]{Definition}
\newtheorem{remark}[theorem]{Remark}

\renewcommand{\P}{\mathbb{P}}

\newcommand{\Ibold}{{\bf I}}

\newcommand{\CBbb}{\mathbb C}

\newcommand{\QBbb}{\mathbb Q}

\newcommand{\ZBbb}{\mathbb Z}

\newcommand{\Acal}{\mathcal A}

\newcommand{\Ccal}{\mathcal C}

\newcommand{\Ecal}{\mathcal E}
\newcommand{\Fcal}{\mathcal F}

\newcommand{\Hcal}{\mathcal H}

\newcommand{\Kcal}{\mathcal K}
\newcommand{\Lcal}{\mathcal L}

\newcommand{\Ocal}{\mathcal O}

\newcommand{\Qcal}{\mathcal Q}

\newcommand{\gfrak}{\mathfrak g}

\newcommand{\Fscr}{\mathscr F}

\DeclareMathOperator{\rank}{rank}

\DeclareMathOperator{\tr}{tr}

\DeclareMathOperator{\sing}{sing}

\DeclareMathOperator{\Gr}{Gr}

\DeclareMathOperator{\supp}{supp}

\DeclareMathOperator{\Quot}{Quot}

\newcommand{\dbar}{\bar\partial}

\newcommand{\lra}{\longrightarrow}

\newcommand{\ch}{{\rm ch}}

\newcommand{\HYM}{\text{\rm\tiny HYM}}

\numberwithin{equation}{section}

\allowdisplaybreaks

\begin{document}
\title[Yang-Mills flow for semistable bundles]{Continuity of the Yang-Mills flow on the set of semistable bundles}

\dedicatory{Dedicated to Duong H. Phong, with admiration, on the occasion of his 65th birthday.}

\author[Sibley]{Benjamin Sibley}
\address{Benjamin Sibley, Simons Center for Geometry and Physics \\
State University of New York \\
Stony Brook, NY 11794-3636,
USA}
\email{bsibley@scgp.stonybrook.edu}

\author[Wentworth]{Richard Wentworth}
\address{Richard Wentworth, Department of Mathematics \\ University of Maryland \\ College Park, MD
20742, USA}
\email{raw@umd.edu}
\urladdr{\href{http://www.math.umd.edu/~raw/}{http://www.math.umd.edu/~raw/}}

\thanks{
R.W.'s research supported in part by NSF grant DMS-1564373. 
 The authors also acknowledge support from NSF grants DMS-1107452, -1107263, -1107367 ``RNMS: GEometric structures And Representation varieties'' (the GEAR Network).}


\keywords{Yang-Mills flow, semistable bundles, Donaldson--Uhlenbeck compactification}
\subjclass[2010]{14D20, 14J60, 32G13, 53C07}

%
\maketitle

\noindent

\section{Introduction}
Let $(X,\omega)$ be a compact K\"ahler manifold of dimension $n$ and $(E,h)\to X$  a $C^\infty$ hermitian vector bundle on $X$. The celebrated theorem of Donaldson-Uhlenbeck-Yau states that if $A$ is an integrable unitary connection  on $(E,h)$ that induces an $\omega$-slope stable holomorphic structure on $E$, then there is a complex gauge transformation $g$ such that  $g(A)$ satisfies the Hermitian-Yang-Mills (HYM) equations. 
The proof in \cite{UhlenbeckYau:86} uses the continuity method applied to a deformation of the  Hermitian-Einstein equations for the metric $h$. The approach in \cite{Donaldson:85,Donaldson:87} deforms the metric using a nonlinear parabolic
 equation, the \emph{Donaldson flow}. Deforming the metric is equivalent to acting by a complex gauge transformation modulo unitary ones, and in this context the Donaldson flow  is equivalent (up to unitary  gauge transformations) to the Yang-Mills flow on the space of integrable unitary connections. The proof in \cite{Donaldson:87}  assumes  that $X$ is a projective algebraic manifold (more precisely, that $\omega$ is a Hodge metric)
 whereas the argument in \cite{UhlenbeckYau:86} does not. The methods of Uhlenbeck-Yau and Donaldson were combined by Simpson \cite{Simpson:88} to prove convergence of the Yang-Mills flow for stable bundles on all compact K\"ahler manifolds. 
The Yang-Mills flow thus defines a map $\Acal^s(E,h)\to M_{\HYM}^\ast(E,h)$ from the space of smooth integrable connections on $(E,h)$ inducing stable holomorphic structures to the moduli space $M^\ast_{\HYM}(E,h)$ of irreducible HYM connections.\footnote{The notion of (semi)stability depends on the choice of K\"ahler class $[\omega]$; however, the class will remain fixed throughout, and we shall suppress this dependency from the notation.} Continuity of this map follows by a comparison of Kuranishi slices (see \cite{FujikiSchumacher:87,Miyajima:89}).

When the holomorphic bundle $\Ecal=(E,\dbar_A)$ is  strictly semistable, then the Donaldson flow fails to  converge unless $\Ecal$ splits holomorphically into a sum of stable bundles (i.e.\ it is  \emph{polystable}). If $n=1$ it is still true, however,  that the Yang-Mills flow converges to a smooth HYM connection on $E$ for any semistable initial condition.
This was proven by Daskalopoulos and R{\aa}de \cite{Daskal:92,Rade:92}. 
Moreover, the holomorphic structure of the limiting connection is isomorphic to the polystable holomorphic bundle $\Gr(\Ecal)$ obtained from  the associated gradation of the Jordan-H\"older filtration of $\Ecal$.
For $n\geq 2$,
there is an obstruction to a smooth splitting into an associated graded bundle, and $\Gr(\Ecal)$ may not be locally free. The new phenomenon of bubbling occurs, and one must talk of convergence \emph{in the sense of Uhlenbeck}, that is, away from a singular set of complex codimension at least $2$ (see Theorem \ref{thm:uhlenbeck} below).
 In \cite{DaskalWentworth:04} (see also \cite{DaskalWentworth:07b}) it was shown for $n=2$ that the Yang-Mills flow converges in the sense of Uhlenbeck to the reflexification $\Gr(\Ecal)^{\ast\ast}$, which is a polystable bundle.  The bubbling locus, which in this case is a collection of points with multiplicities,  is precisely the set where $\Gr(\Ecal)$ fails to be locally free \cite{DaskalWentworth:07a}. The extension of these results in higher dimensions was achieved in \cite{Sibley:15,SibleyWentworth:15}. Here, even the reflexified associated graded sheaf may fail to be locally free, and one must use the notion of an \emph{admissible} HYM connection introduced by Bando and Siu \cite{BandoSiu:94}.  Convergence of the flow to the associated graded sheaf for semistable bundles in higher dimensions was independently proven by Jacob \cite{Jacob:15}.

In a different direction, a compactification of $M_{\HYM}^\ast$
 was proposed by Tian in \cite{Tian:00} and further studied 
 in \cite{TianYang:02}. This may be viewed as a higher dimensional version of the Donaldson-Uhlenbeck compactification of ASD connections on a smooth manifold of real dimension $4$ (cf.\ \cite{FreedUhlenbeck:84,DonaldsonKronheimer:90}).
It is based on a finer analysis of the bubbling locus for limits of HYM connections that is similar to the one carried out for harmonic maps by Fang-Hua Lin \cite{Lin:99}. More precisely, Tian proves that the top dimensional stratum is rectifiable and calibrated by $\omega$ with integer multiplicities, and as a consequence of results of King \cite{King:71} and Harvey-Shiffman \cite{HarveyShiffman:74}, it represents an analytic cycle.
 The compactification  is  then defined by adding ideal points containing in addition to an admissible HYM connection the data of a codimension $2$ cycle in an appropriate cohomology class (see Section \ref{sec:uhlenbeck}). At least when $X$ is projective, 
 the space $\widehat M_{\HYM}$ of ideal HYM connections is a compact topological space (Hausdorff), and the compactification of $M^\ast_{\HYM}$ is obtained by taking its closure $\overline M_{\HYM}\subset \widehat M_{\HYM}$. 
Under this assumption, we recently showed,
    in collaboration with Daniel Greb and Matei Toma, that $\overline M_{\HYM}$ admits the structure of a seminormal complex algebraic space \cite{GSTW:18}.

The purpose of this note is to point out  the compatibility of this construction with the Yang-Mills flow. For example, in the case of a Riemann surface, the flow defines a continuous deformation retraction of the entire semistable stratum  onto the moduli space of semistable bundles.  This is precisely what is to be expected from Morse theory (see \cite{AtiyahBott:82}).  In higher dimensions, as mentioned above, bubbling along the flow needs to be accounted for. 
The  result is the following.

\begin{theorem*} 
Let $(E,h)$ be a hermitian vector bundle over a compact K\"ahler manifold 
$(X,\omega)$ with $[\omega]\in H^2(X,\ZBbb)$. Let $\Acal^{ss}(E,h)$ be the set of semistable integrable unitary connections on $(E,h)$ with the smooth topology (see Section \ref{sec:uhlenbeck}).
 Then the Yang-Mills flow defines a continuous map 
 \begin{equation}\label{eqn:F}
 \Fscr: \Acal^{ss}(E,h)\to \widehat M_{\HYM}(E,h)\ .
 \end{equation}
  In particular, the restriction of $\Fscr$ gives a continuous map $\overline{\Acal^s}(E,h)\to \overline M_{\HYM}(E,h)$, where \break $\overline{\Acal^s}(E,h)\subset \Acal^{ss}(E,h)$ is the closure of $\Acal^s(E,h)$ in the smooth topology.
\end{theorem*}

The proof of the Main Theorem is a consequence of the work in \cite{GSTW:18}, with small modifications.
For the case of K\"ahler surfaces, 
this result was claimed in \cite[Thm.\ 2]{DaskalWentworth:07a}. Unfortunately, there is an error in the proof of Lemma 8 of that paper, and hence also in the proof of Theorem 2. The Main Theorem above  validates the statement in \cite[Thm.\ 2]{DaskalWentworth:07a}, at least in the projective case. We do not know if the result holds when $X$ is only K\"ahler.  The advantage of projectivity is that  a twist of the bundle is generated by global  holomorphic sections.  These behave well with respect to Uhlenbeck limits and provide a link between the algebraic geometry of geometric invariant theory quotients and the analytic compactification. We review this in Section \ref{sec:quot} below.

\section{Uhlenbeck limits and admissible HYM connections} \label{sec:uhlenbeck}
In this section we briefly review the compactification of $M_{\HYM}^\ast(E,h)$ by ideal HYM connections.
As in the introduction, let
  $(E,h)$ be a hermitian vector bundle on  a compact K\"{a}hler manifold $(X, \omega)$ of dimension $n$, and let $\gfrak_E$ denote the bundle of skew-hermitian endomorphisms of $E$. The space $\Acal(E,h)$ of $C^\infty$ unitary connections on $E$ is an affine space over $\Omega^1(X,\gfrak_E)$, and we endow it with the smooth topology.
  A connection $A\in\Acal(E,h)$ is called
   \emph{integrable} if its curvature form $F_A$ is of type (1,1).
Let $\Acal^{1,1}(E,h)$ denote the set of  integrable unitary connections on $(E,h)$. 
Then  $\Acal^{1,1}(E,h)\subset\Acal(E,h)$ inherits a  topology as a closed subset.
The locus $\Acal^s(E,h)$  of \emph{stable} holomorphic structures is open in $\Acal^{1,1}(E,h)$ (cf.\ \cite[Thm.\ 5.1.1]{LubkeTeleman:95}). Under the assumption that $\omega$ is a Hodge metric we shall prove below that
 the subset $\Acal^{ss}(E,h)$ of \emph{semistable} holomorphic structures is also open in $\Acal^{1,1}(E,h)$ (see Corollary \ref{cor:open}).

   We call the contraction $\sqrt{-1}\Lambda F_A$ of $F_A$ with the K\"ahler metric the \emph{Hermitian-Einstein tensor}.  It is a hermitian endomorphism of $E$.
The  key definition is the following (cf.\ \cite{BandoSiu:94} and \cite[Sect.\ 2.3]{Tian:00}).
\begin{definition}\label{def:admissible}
   An \emph{admissible connection} is a pair $(A,S)$ where 
 \begin{enumerate}
 \item  $S\subset X$ is a closed subset of finite Hausdorff $(2n-4)$-measure;
 \item $A$ is a smooth integrable unitary connection on $E\bigr|_{X\backslash S}$;
 \item $\int_{X\backslash S} |F_A|^2\, dvol_X < +\infty$;
 \item $\sup_{X\backslash S}  | \Lambda F_A| < +\infty$.
 \end{enumerate}
 An admissible connection is called \emph{admissible HYM} if there is a constant $\mu$ such that  $\sqrt{-1}\Lambda F_A=\mu\cdot \Ibold$ on $X\backslash S$.
 \end{definition}

The fundamental weak compactness result is the following.

\begin{theorem}[Uhlenbeck \cite{UhlenbeckPreprint}] \label{thm:uhlenbeck}
Let  $A_i$ be a  sequence of smooth integrable connections on $X$ with uniformly bounded Hermitian-Einstein tensors.
 Then for any $p>n$ there is
\begin{enumerate}
\item a subsequence $($still denoted $A_{i}$$)$,
\item a closed subset $S_{\infty}\subset X$ of  finite $(2n-4)$-Hausdorff measure,
\item a connection $A_\infty$ on a hermitian bundle $E_\infty\to X\backslash S_{\infty}$, and
\item local isometries $E_\infty \simeq E$ on compact subsets of $X\backslash S_{\infty}$
\end{enumerate}
such that with respect to the local isometries, and  modulo unitary gauge equivalence, $A_{i}\to A_\infty$ weakly in $
L^p_{1,loc}(X\backslash S_{\infty})$.
\end{theorem}

We call the limiting connection $A_\infty$ an \emph{Uhlenbeck limit}.
 The set 
$$
S_\infty=
\bigcap_{\sigma_0\geq \sigma>0}\Bigl\{
	x\in X \mid \liminf_{i\to\infty} \sigma^{4-2n}\int_{B_\sigma(x)}|F_{A_i}|^2 \frac{\omega^n}{n!}\geq \varepsilon_0\Bigr\}\ ,
$$
where $\sigma_0$ and $\varepsilon_0$ are universal constants depending only on the geometry of $X$,
is called the 
\emph{(analytic) singular set}.

For the definition of a gauge theoretic compactification 
more structure is needed. This is provided
 by the following, which is a consequence of work of Tian \cite{Tian:00} and Hong-Tian \cite{HongTian:04}. 

\begin{proposition} \label{prop:admissible} 
The Uhlenbeck limit of a sequence of smooth HYM connections on $(E,h)$ is an admissible HYM connection. Moreover, the corresponding singular set $S_\infty$ is a holomorphic subvariety of codimension at least $2$. The same is true for Uhlenbeck limits of sequences along the Yang-Mills flow.
\end{proposition}

To be more precise, there is a decomposition $S_\infty=|\Ccal_\infty|\cup S(A_\infty)$, where
\begin{equation}\label{eqn:singset}
S(A_\infty):=\biggl\{ x\in X \,\biggr |\, \lim_{\sigma\downarrow
0}\sigma^{4-2n}\int_{B_{\sigma}(x)}\left\vert F_{A_\infty}\right\vert
^{2}\frac{\omega^n}{n!}\neq 0\ \biggr\}.
\end{equation}
has codimension $\geq 3$, and $|\Ccal_\infty|$ is the support of a codimension 2 cycle $\Ccal_\infty$. The cycle appears as the limiting current of the Yang-Mills energy densities, just as in the classical approach of Donaldson-Uhlenbeck in real dimension $4$. This structure motivates the following

\begin{definition}[{\cite[Def.\ 3.15]{GSTW:18}}] \label{def:ideal-connection}
An {\em ideal HYM connection}  is a triple  $(A,\Ccal, S(A))$ satisfying the following conditions:
\begin{enumerate}
\item $\Ccal$ is an $(n-2)$-cycle on $X$;
\item the pair $(A, |\Ccal| \cup S(A))$ is an  admissible HYM connection on the hermitian vector bundle $(E,h)\to X$, where $S(A)$ is 
given as in eq.\ \eqref{eqn:singset};
\item $[\ch_2(A)]=\ch_2(E)+[\Ccal]$, in $H^4(X,\QBbb)$;
\end{enumerate} 
\end{definition}

Here we have denoted by $\ch_{2}(A)$ the $(2,2)$-current given by%
\begin{equation*}
\ch_{2}(A)(\Omega):=-\frac{1}{8\pi^2}\int_{X}\tr(F_{A}\wedge F_{A})\wedge\Omega\ ,
\end{equation*}%
for smooth $(2n-4)$-forms $\Omega$.
This is well defined by Definition \ref{def:admissible} (3), and in
 \cite[Prop.\ 2.3.1]{Tian:00} it is shown to be a closed current. It thus
  defines a cohomology class as above. By  \cite{BandoSiu:94}, there is a polystable reflexive
sheaf $\mathcal{E}$ extending the holomorphic bundle $(E|_{X\backslash Z\cup
S(A)},\overline{\partial }_{A})$. The singular set $\sing(\Ecal)$ of $\Ecal$, that is, the locus where $\Ecal$ fails to be locally free, coincides with $S(A)$ (see \cite[Thm.\ 1.4]{TianYang:02}). By the proof of \cite[Prop.\ 3.3]{SibleyWentworth:15}, $\ch_{2}(A)$ represents the
class $\ch_{2}(\mathcal{E})$. Thus we may alternatively regard an ideal connection as a
pair $(\mathcal{E},\mathcal{C})$, where $\Ecal$ is a reflexive sheaf, $\Ccal$ is a codimension $2$ cycle with
$
\ch_{2}(\mathcal{E})=\ch_{2}(E)+[\mathcal{C}]
$, and where the underlying smooth bundle of $\Ecal$ on the complement of $|\Ccal|\cup \sing(\Ecal)$ is isomorphic to $E$.
See \cite[Sec.\ 3.3]{GSTW:18} for more details. 

There is an obvious notion of gauge equivalence of ideal HYM connections. The main result is the following.

\begin{theorem} \label{thm:ideal-convergence}
Assume $\omega$ is a Hodge metric.
Let $(A_i, \Ccal_i, S(A_i))\in \widehat M_\HYM(E,h)$.
Then there is a subsequence $($also denoted by $\{i\}$$)$, and an ideal HYM connection $(A_\infty, \Ccal_\infty, S(A_\infty))$
such that $\Ccal_i$ converges to a subcycle of $\Ccal_\infty$, and (up to gauge transformations) $A_i\to A_\infty$ in $C^\infty_{loc}$ on $X\backslash (|\Ccal_\infty|\cup S(A_\infty))$. Moreover, 
\begin{equation} \label{eqn:currents-converge}
\ch_2(A_i)-\Ccal_i\lra \ch_2(A_\infty)-\Ccal_\infty
\end{equation}
in the mass norm; in particular, also
in the sense of currents.
\end{theorem}
For more details we refer to \cite{Tian:00,TianYang:02,GSTW:18}.

\section{The method of holomorphic sections} \label{sec:quot}
Admissibility of a connection is precisely the correct analytic notion to make contact with complex analysis. Bando \cite{Bando:91} and Bando-Siu \cite{BandoSiu:94} show that bundles with admissible connections admit sufficiently many local holomorphic sections to prove coherence of the sheaf of $L^2$-holomorphic sections.
This local statement only requires the K\"ahler condition.
The key difference between the projective vs.\ K\"ahler case is, of course, the abundance of global holomorphic sections. 
These provide a link between the algebraic and analytic moduli. They are also well-behaved with respect to limits. The technique described here mimics that introduced by Jun Li in \cite{Li:93}.

We henceforth assume $[\omega]\in H^2(X,\ZBbb)$.
Let $L\to X$ be a complex line bundle with  $c_1(L)=[\omega]$. 
Define the numerical invariant:
\begin{equation} \label{eqn:hilbert-poly}
\tau_E(m):=\int_X \ch(E\otimes L^{m}){\rm td}(X) \ .
\end{equation}
 Since $\omega$ is a $(1,1)$ class, $L$ may be endowed with a holomorphic structure $\Lcal$ making it
  the ample line bundle defining the polarization of $X$. 
We also fix a hermitian metric on $L$ with respect to which the Chern connection of $\Lcal$ has curvature $-2\pi i\omega$.
  Use the following notation: $\Ecal(m):= \Ecal\otimes\Lcal^m$.  The key property we exploit is the following, which is a consequence of Maruyama's boundedness result \cite{Maruyama:81}, as well as the Hirzebruch-Riemann-Roch theorem.

\begin{proposition} \label{prop:maruyama}
There is $M\geq 1$ such that for all $m\geq M$ and all $A\in \Acal^{ss}(E,h)$, if $\Ecal=(E,\dbar_A)$ then
the bundle $\Ecal(m)$ is globally generated and all higher cohomology groups vanish. In particular, $\dim H^0(X, \Ecal(m))=\tau_E(m)$ for $m\geq M$.
\end{proposition}

In the following, we shall assume $m$ has been fixed sufficiently large (possibly larger than in the previous proposition).
Fix a vector space $V$ of dimension $\tau_E(m)$, and let 
\begin{equation} \label{eqn:H}
\Hcal=V\otimes \Ocal_X(-m)\ .
\end{equation}
 The \emph{Grothendieck Quot scheme} $\Quot(\Hcal, \tau_E)$ is a projective scheme parametrizing isomorphism classes of quotients $\Hcal \to \Fcal\to 0$, where $\Fcal\to X$ is a coherent sheaf with Hilbert polynomial $\tau_E$ \cite{Grothendieck:61,AltmanKleiman:80}.
Proposition \ref{prop:maruyama} states that there is a uniform $m$ such that for every $A\in \Acal^{ss}(E,h)$ there is a quotient $\Hcal\to \Ecal\to 0$ in $\Quot(\Hcal, \tau_E)$ with $\Ecal\simeq (E,\dbar_A)$. The next result begins the comparison between Uhlenbeck limits and limits in $\Quot(\Hcal, \tau_E)$.

\begin{proposition} \label{prop:limit_sections}
Let $\{A_i\}\subset \Acal^{ss}(E,h)$, and suppose $A_i\to A_\infty$ in the sense of Uhlenbeck (Theorem \ref{thm:uhlenbeck}), and assume uniform bounds on the Hermitian-Einstein tensors. 
Then there is a quotient $\Hcal\to \Fcal_\infty\to 0$ in $\Quot(\Hcal, \tau_E)$ and an inclusion $\Fcal_\infty\hookrightarrow \Ecal_\infty$ such that $\Fcal_\infty^{\ast\ast}\simeq \Ecal_\infty$. 
\end{proposition}

The proof of this result for sequences of HYM connections is in \cite[Sec.\ 4.2]{GSTW:18}, but the proof there works as well under the weaker assumption of bounded Hermitian-Einstein tensor. Indeed, the first key point is the application of the Bochner formula to obtain uniform bounds on $L^2$-holomorphic sections. The precise statement is that if $s\in H^0(X,\Ecal(m))$ then there is a constant $C$ depending only on the geometry of $X$, $m$, and the uniform bound on the Hermitian-Einstein tensor, such that $\sup_X|s|\leq C \Vert s\Vert_{L^2}$.
  In this way one can extract a convergent subsequence of orthonormal sections to obtain a map $q_\infty: \Hcal \to \Ecal_\infty$. The limiting sections may no longer form a basis of $H^{0}(X,\Ecal_{\infty}(m))$, nor necessarily do they
   generate the fiber of $\Ecal_\infty$. Remarkably, though, it is still the case that the rank of the image sheaf $\widetilde \Ecal_\infty\subset\Ecal_\infty$ of $q_\infty$
    agrees with $\rank(E)$ and has Hilbert polynomial $\tau_E$ (for this one may have to twist with a further power of $\Lcal$). 
    In fact, the quotient sheaf $\mathcal{T}_{\infty }=\mathcal{E}_{\infty }/%
\widetilde{\mathcal{E}}_{\infty }$ turns out to be supported in complex
codimension $2$ (the first Chern class is preserved under Uhlenbeck limits). 
Hence, in particular, $(\widetilde \Ecal_\infty)^{\ast\ast}\simeq \Ecal_\infty$.  See \cite[proof of Lemma 4.3]{GSTW:18}
for more details.

The second ingredient in the proof is the fact that $\Quot(\Hcal, \tau_E)$ is  compact in the analytic topology. Hence, after passing to a subsequence, we may assume the $q_i$ converge. Convergence in
 $\Quot(\Hcal, \tau_E)$ means the following: there is a convergent sequence of quotients $\Fcal_i\to \Fcal_\infty$ and isomorphisms $f_i$ making the following diagram commute
\begin{equation*}
\begin{split}
\xymatrix{
\Hcal\ar@{=}[d] \ar[r] & \Fcal_i \ar[d]^{f_i} \ar[r] & 0 \\
\Hcal\ar[r]^{q_i}&\Ecal_i\ar[r] & 0
}
\end{split}
\end{equation*}
The proof is completed, as in \cite[Lemma 4.4]{GSTW:18},  by showing that $\Fcal_\infty\simeq \widetilde \Ecal_\infty$. 
The crucial point that is used in showing this is that the two sheaves are quotients of $\Hcal$ with  the same Hilbert polynomial.

\section{Analytic cycles and the blow-up set} \label{sec:cycle}
In the case of the stronger notion of convergence of Uhlenbeck-Tian, we go one step further  and identify the cycle associated to the sheaf $\Fcal_{\infty}$ with the cycle $\Ccal_{\infty}$ that arises from bubbling of the connections. The candidate is the following:
 for any torsion-free sheaf $\Fcal\to X$, define a codimension $2$ cycle $\Ccal_\Fcal$
from the top dimensional stratum of the support of $\Fcal^{\ast\ast}/\Fcal$. See for example \cite[Sec.\ 2.5.3]{GSTW:18}.

\begin{proposition} \label{prop:cycles}
Let $A_{i}$ be a sequence of connections as in Proposition \ref{prop:limit_sections}, 
and suppose furthermore that they 
converge to
an ideal connection $(A_{\infty },\mathcal{C}_{\infty },S(A_{\infty }))$ in the sense of Theorem \ref{thm:ideal-convergence}. Let $\Hcal\to \Fcal_\infty$ be as in the statement of Proposition \ref{prop:limit_sections}. 
Then  $\mathcal{C}_{\infty }=
\mathcal{C}_{\mathcal{F}_{\infty }}$.
\end{proposition}

 The proof of this result follows from the discussion in \cite[Sec.\ 4.3]{GSTW:18} (see in
particular  Prop.\ 4.7). Although the result there is
stated for sequence of HYM connections, this is required only to obtain the same sup-norm inequality on the global
sections of $\mathcal{E}(m)$ that was used to obtain Proposition \ref{prop:limit_sections}.
Thus, the uniform bound on the Hermitian-Einstein tensor suffices. 

Let us sketch the argument. 
The first key point is that $[\mathcal{C}_{\mathcal{F}
_{\infty }}]=[$ $\mathcal{C}_{\infty }]$ in rational cohomology. Indeed, the connection $A_{\infty }$ is defined on the smooth locus of 
 the sheaf $\mathcal{F}_{\infty }^{\ast \ast }$ and is smooth there,
and $\ch_{2}(A_{\infty })$ defines a closed current (see Section \ref{sec:uhlenbeck}). 
It then follows as in the proof of \cite[Prop.\ 3.3]{SibleyWentworth:15}
 that $
[\ch_{2}(A_{\infty })]=\ch_{2}(\mathcal{F}_{\infty }^{\ast \ast })$. The
exact sequence%
\begin{equation*}
0\lra \mathcal{F}_{\infty }\lra \mathcal{F}_{\infty }^{\ast
\ast }\lra \mathcal{T}_{\infty }\lra 0\ ,
\end{equation*}%
implies that 
\begin{equation*}
\lbrack \ch_{2}(A_{\infty })]=\ch_{2}(\mathcal{F}_{\infty })+\ch_{2}(\mathcal{T}%
_{\infty })=\ch_{2}(E)+[\mathcal{C}_{\mathcal{F}_{\infty }}]\ ,
\end{equation*}%
where in the second inequality we have used the fact that the Chern classes of $%
\mathcal{F}_{\infty }$ are the same as those of $E$, and Proposition 3.1 of \cite{SibleyWentworth:15}
(see also the latest arxiv version of this reference). By the convergence of
the currents in Theorem \ref{thm:ideal-convergence} and  Chern-Weil theory, we have 
\begin{equation*}
\ch_{2}(E)+[\mathcal{C}_{\infty }]=[\ch_{2}(A_{i})]+[\mathcal{C}_{\infty
}]=[\ch_{2}(A_{\infty })].
\end{equation*}
Combining these two equalities gives the statement. 

What remains to be shown is that given any irreducible component $
Z\subset \supp(\mathcal{T}_{\infty })$, for the associated multiplicity $
m_{Z}$ as defined in \cite[Sec.\ 2.5.3]{GSTW:18}, we have an equality%
\begin{equation*}
m_{Z}=\lim_{i\rightarrow \infty }\frac{1}{8\pi ^{2}}\int_{\Sigma
}\tr(F_{A_{i}}\wedge F_{A_{i}})-\tr(F_{A_{\infty }}\wedge F_{A_{\infty }})\ ,
\end{equation*}%
where $\Sigma $ is a generic real $4$-dimensional slice
intersecting $Z$ transversely  in a single smooth point. The point is that if $Z$ is
contained in the support $|\mathcal{C}_{\infty }|$ then it must be equal to
one of the irreducible components. In this case, the number on the right
hand side of the equality above is exactly the multiplicity of this component in
the cycle $\mathcal{C}_{\infty }$, and otherwise this number is zero, (see
 \cite[Lemma 3.13]{GSTW:18} and  \cite[Lemma 4.1]{SibleyWentworth:15} and again note that the proof is
completely general). If the equality holds, this number cannot be zero,
since $m_{Z}$ is strictly positive by definition, and therefore $Z$ must be
a component of $\mathcal{C}_{\infty }$, and the multiplicities agree. Since $%
\mathcal{C}_{\infty }$ and $\mathcal{C}_{\mathcal{F}_{\infty }}$ are equal
in cohomology, there can be no other irreducible components of $\mathcal{C}%
_{\infty }$, and so $\mathcal{C}_{\infty }=$ $\mathcal{C}_{\mathcal{F}%
_{\infty }}$. For more details, see the proof of  \cite[Prop.\ 4.7]{GSTW:18}).

\begin{remark}
It should be emphasized that Proposition \ref{prop:cycles} does not claim that the support of $\Fcal_\infty^{\ast\ast}/\Fcal_\infty$ coincides with the full bubbling locus $|\Ccal_\infty|\cup S(A_\infty)$; only the top dimensional strata are necessarily equal. This differs from what occurs, for example, along the Yang-Mills flow (see \cite[Thm.\ 1.1]{SibleyWentworth:15}).
 It would be interesting to understand the behavior of the higher codimensional pieces from this perspective. There are recent examples due to Chen-Sun indicating that this should be subtle (see \cite{ChenSun:18b,ChenSun:18a}).
\end{remark}

\section{A remark on the topology of the Quot scheme}
 In this section we consider the relationship between the Quot scheme $\Quot(\Hcal,\tau_E)$ discussed in Section \ref{sec:quot}, and the infinite dimensional space $\Acal^{1,1}(E,h)$ of integrable connections. 
 We are thus interested in the points in $\Quot(\Hcal,\tau_E)$ where the quotient sheaf is locally free and has underlying $C^\infty$ bundle isomorphic to $E$. Such a point corresponds to an isomorphism class of holomorphic structures on $E$, or equivalently, to a complex gauge orbit in $\Acal^{1,1}(E,h)$. Conversely, a connection $A\in \Acal^{1,1}(E,h)$ gives a holomorphic bundle which, provided $m$ is sufficiently large, can be realized as a quotient. We wish to show that this correspondence between complex gauge orbits in $\Acal^{1,1}(E,h)$ and points in $\Quot(\Hcal,\tau_E)$ can be made continuous in the respective topologies. Since the complex gauge
  orbit space in $\Acal^{1,1}(E,h)$ is non-Hausdorff in general, we will lift to a map from open sets in $\Acal^{1,1}(E,h)$ itself.

This gives rise to the following notion. 
Let $U\subset \Acal^{1,1}(E,h)$. We call $\sigma : U\to \Quot(\Hcal,\tau_E)$ a \emph{classifying map}
if the quotient $\sigma(A)$ is a holomorphic bundle isomorphic to $(E,\dbar_A)$.
Recall from Section \ref{sec:quot} that the bundle $\Hcal$  depends on a sufficiently large choice of $m$, which we omit from the notation. Then the result is the following. 

\begin{theorem} \label{thm:classifying}
Fix $A_0\in \Acal^{1,1}(E,h)$. Then for $m$ sufficiently large (depending on $A_0$), there is an open neighborhood $U\subset\Acal^{1,1}(E,h)$ of  $A_0$ and a continuous classifying map $\sigma : U\to \Quot(\Hcal,\tau_E)$. On $\Acal^{ss}(E,h)$, the twist $m$ may be chosen uniformly. 
\end{theorem}

Throughout the proof, as in Section \ref{sec:quot}, we fix a hermitian structure $h_L$ on $L$ such that the curvature of the Chern connection of $(\Lcal, h_L)$ defines a K\"ahler metric $\omega$ on $X$.

\begin{proof}
Let $d(m,n)=\tau_E(m)\cdot \dim H^0(X,\Lcal^n)$. For $n>>1$,  $\Quot(\Hcal,\tau_E)$ is embedded in the Grassmannian $G(d(m,n), \tau_E(m+n))$ of $\tau_E(m+n)$-dimensional quotients of $\CBbb^{d(m,n)}$. More precisely, suppose  $q: \Hcal\to \Ecal$ is a point in $\Quot(\Hcal,\tau_E)$, and let $\Kcal=\ker q$. There is a sufficiently large  $n$  (uniform over the whole Quot scheme) such that
\begin{equation} \label{eqn:vanishing}
H^i(X,\Kcal(m+n))=H^i(X,\Ecal(m+n))=\{0\}\ ,\ i\geq 1
\end{equation}
(cf.\ \cite[Lemmas 1.7.2 and 1.7.6]{HuybrechtsLehn:10}).
We therefore have a short exact sequence:
\begin{equation} \label{eqn:quotient}
0\lra H^0(X, \Kcal(m+n))\lra H^0(X, \Hcal(m+n))\lra H^0(\Ecal(m+n))\lra 0
\end{equation}
Since 
the middle term 
  has dimension $d(m,n)$, 
and by \eqref{eqn:vanishing} the last term has dimension $\tau_E(m+n)$,
we obtain a point in $G(d(m,n), \tau_E(m+n))$. For $n$ sufficiently large, this is an embedding.

Given $A_0\in \Acal^{1,1}(E,h)$, $\Ecal_0=(E,\dbar_{A_0})$, 
choose $m$ such that $\Ecal_0(m)$ is globally generated and has no higher cohomology. 
Set $V_0=H^0(X, \Ecal_0(m))$. Then $\dim V_0=\tau_E(m)$. Given $s\in V_0$, $f\in \Ocal_X$, the map 
$${\rm ev} : V_0\otimes \Ocal_X\lra \Ecal_0(m) : s\otimes_\CBbb f\mapsto fs $$
realizes $\Ecal_0(m)$ as a quotient of $V_0\otimes \Ocal_X$. After twisting back by $\Ocal_X(-m)$, we have a quotient $\Hcal\to \Ecal_0\to 0$.

 For $A\in U$ (the open set $U$ remains to be specified) in order to realize $\Ecal =(E,\dbar_A)$ as a quotient of $\Hcal$, it suffices to give an isomorphism of $V_0$ with $V_A=\ker\dbar_A\subset \Gamma(E\otimes L^m)$, for then $\Ecal$ is obtained through this isomorphism followed by evaluation ${\rm ev}$ as above. Note here that we assume already that $U$ has been chosen sufficiently small so that $\Ecal(m)$ is globally generated and has no higher cohomology.
This is the first condition on $U$, and it can be arranged by semicontinuity of cohomology.

 On $\Gamma(E\otimes L^m)$ we have an $L^2$-inner product. Since $V_A$ and $V_0$ are subspaces of $\Gamma(E\otimes L^m)$, we can define a
 map by orthogonal projection $\pi_A: V_A\to V_0$. Let us write this explicitly. For $s\in V_A$, let $\pi_A(s)=s_0=s+u_s$, where $u_s\in V_0^\perp$. We require $\dbar_{A_0}s_0=0$, or $\dbar_{A_0}(s+u_s)=0$. If we write $\dbar_A=\dbar_{A_0}+a$, $a\in \Omega^{0,1}(X,\gfrak_E)$, then the above is $\dbar_{A_0}u_s=as$. Let $G_0$ be the Green's operator for the $\dbar_{A_0}$ laplacian acting on $\Omega^{0,1}(E\otimes L^m)$. In general, the Green's operator inverts the laplacian up to projection onto the orthogonal complement of 
 the harmonic forms in $\Omega^{0,1}(X,E\otimes L^{m})$. We have 
 assumed vanishing of $H^1(X,\Ecal_0(m))$, so in our case $G_0$ is a genuine inverse. Set: $u_s=\dbar_{A_0}^\ast G_0(as)$.  Then
 $$
\dbar_{A_0}u_s=\dbar_{A_0}\dbar_{A_0}^\ast G_0(as)=\square_{A_0}G_0(as)+\dbar_{A_0}^\ast\dbar_{A_0}G_0(as)=as\ ,
 $$
 as desired. Here, 
  we have used the fact that $\dbar_{A_0}G_0=G_0\dbar_{A_0}$, and that, by the integrability of $\dbar_{A}$ and $s\in V_{A}$, $\dbar_{A_0}(as)=0$. Notice that this definition of $u_s$ guarantees that it is orthogonal to $V_0$.  Moreover, since $\dbar_{A_0}^\ast G_0$ is a bounded operator, we have an estimate
\begin{equation} \label{eqn:u-est}
\Vert u_s\Vert_{L^2}\leq B \Vert as\Vert_{L^2} \leq B(\sup|a|) \Vert s\Vert_{L^2}\ .
\end{equation}
In particular, for $\sup|a|$ sufficiently small, $\Vert\pi_A(s)\Vert_{L^2}\geq (1/2)\Vert s\Vert_{L^2}$, so $\pi_A$ is injective and therefore an isomorphism. 
The classifying map  is then defined as the quotient:
$$ \sigma(A) :
\Hcal\xrightarrow{\pi_A^{-1}\otimes id}V_A\otimes \Ocal_X(-m)\xrightarrow{\rm ev} \Ecal\lra 0\ .
$$

 It remains to show that $\sigma$  is continuous.  We begin with a few preliminaries. 
For $s\in \Gamma(E\otimes L^m)$, let  
$$\widetilde \pi_A : \Gamma(E\otimes L^m)\lra \Gamma(E\otimes L^m) :  s\mapsto s+\dbar_{A_0}^\ast G_0(as)\ ,
$$
 so that $\widetilde \pi_A$ restricted to $V_A$ is $\pi_A$. Again using that $\dbar_{A_0}^\ast G_0$ is a bounded operator, we have
 $$
\Vert (\widetilde \pi_{A_1}-\widetilde \pi_{A_2})s\Vert_{L^2}=
\Vert\dbar_{A_0}^\ast G_0((a_1-a_2)s)\Vert_{L^2}\leq B\Vert (a_1-a_2)s\Vert_{L^2}\leq B\sup|a_1-a_2|\Vert s\Vert_{L^2}\ .
 $$
 It follows that $\widetilde \pi_A$ is continuous in $A$. By the argument following \eqref{eqn:u-est}, it is also uniformly invertible for $A\in U$, with $\Vert\widetilde \pi_A^{-1}\Vert\leq 2$. Hence,
 \begin{align*}
(\widetilde \pi_{A_1}^{-1}-\widetilde \pi_{A_2}^{-1})s
&= \widetilde \pi_{A_1}^{-1} (\widetilde \pi_{A_2}-\widetilde \pi_{A_1})
 \widetilde \pi_{A_2}^{-1}s \\
 \Vert (\widetilde \pi_{A_1}^{-1}-\widetilde \pi_{A_2}^{-1})s\Vert_{L^2} &\leq 4 \Vert
 \widetilde \pi_{A_1}-\widetilde \pi_{A_2}\Vert\cdot \Vert s\Vert_{L^2}\leq 4B\sup|a_1-a_2|\Vert s\Vert_{L^2}\ .
 \end{align*}
 We conclude that the map $\pi_A^{-1} : V_0\to \Gamma(E\otimes L^m)$, whose image is $V_A$, is continuous for $A\in U$, and in fact satisfies an estimate:
 \begin{equation} \label{eqn:inverse-estimate}
\Vert (\pi_{A_1}^{-1}- \pi_{A_2}^{-1})s_0\Vert_{L^2}
\leq 4B\sup|a_1-a_2|\Vert s_0\Vert_{L^2}\ ,
 \end{equation}
 for all $s_0\in V_0$.

 The second ingredient we shall need is the following. We may assume that there is a uniform bound on the Hermitian-Einstein tensors for each $A\in U$. It follows as in Section \ref{sec:quot} that we  have an estimate:
$\sup|s|\leq C\Vert s\Vert_{L^2}$, for all $s\in V_A$. Hence, 
\begin{equation} \label{eqn:sup-bound}
\sup|\pi_A^{-1}(s_0)|\leq C\Vert \pi_A^{-1}(s_0)\Vert_{L^2}\leq 2C\Vert s_0\Vert_{L^2}\ ,
\end{equation}
for all $s_0\in V_0$.

Finally, let $s_0\in V_0$ and $f\in H^0(X,\Lcal^n)$. Since $f$ is holomorphic any norm is bounded by its $L^2$ norm. Using \eqref{eqn:inverse-estimate} and \eqref{eqn:sup-bound}, there is a constant $C_1>0$ such that:
\begin{align}
\Vert f(\pi_{A_1}^{-1}-\pi_{A_2}^{-1})s_0\Vert_{L^2}^2
&\leq \Vert f\Vert_{L^4}^2\Vert (\pi_{A_1}^{-1}-\pi_{A_2}^{-1})s_0\Vert_{L^4}^2 
\leq
C_1\Vert f\Vert_{L^2}^2 \Vert (\pi_{A_1}^{-1}-\pi_{A_2}^{-1})s_0\Vert_{L^2} \Vert s_0\Vert_{L^2} \notag\\
&\leq
C_1 (\sup|a_1-a_2|) \Vert s_0\Vert_{L^2}^2 \Vert f\Vert_{L^2}^2\ .\label{eqn:L4}
\end{align}

 To prove continuity of $\sigma$,
  we show that the corresponding quotients \eqref{eqn:quotient} vary continuously in the Grassmannian for $A\in U$. 
First, notice that
\begin{equation*} \label{eqn:Hmn}
H^0(X,\Hcal(m+n))\simeq V_0\otimes H^0(X,\Lcal^n)\ ,
\end{equation*}
On $V_0\otimes H^0(X,\Lcal^n)$,
 we choose the tensor product metric of the $L^2$ metrics on $V_0$ and $H^0(X,\Lcal^n)$.
The map induced by $\sigma$ is described as follows: for each $A\in U$ we have
$$
T_A :  V_0\otimes H^0(X,\Lcal^m)\lra \Gamma(E\otimes L^{m+n}) : s\otimes_\CBbb f\mapsto \pi_A^{-1}(s)\otimes_{\Ocal_X}f
$$
with image $H^0(X, \Ecal(m+n))$.
    Moreover, it follows as in  \eqref{eqn:L4} that $T_A$ is continuous in $A$ for $A\in U$.

 Let $P_A$ denote the orthogonal projection to $\ker T_A$. The topology of the Grassmannian may defined through projection operators, so it suffices to show that 
$P_A$ is continuous in $A\in U$. Because the dimensions of the kernels of $T_A$ are constant on $U$, this reduces to showing that for any sequence $A_j\to A$
and  $s_j\in \ker T_{A_j} \subset V_0\otimes H^0(X,\Lcal^n)$, $\Vert s_j\Vert=1$, there is a 
subsequence such that $s_j\to s\in \ker T_A$. Indeed, if this is the case we may choose an orthonormal basis of such sections, $\{s_j^\alpha\}$, so that for any $s\in V_0\otimes H^0(X,\Lcal^n)$,
$$
P_{A_j} s=\sum_\alpha \langle s, s_j^\alpha\rangle s_j^\alpha\ ,
$$
and the right hand side converges to $P_A s$, and so $\Vert P_{A_j}-P_A\Vert\to 0$. By finite dimensionaliy of $V_0\otimes H^0(X,\Lcal^n)$, we may assume $s_j\to s$ for some $s\in V_0\otimes H^0(X,\Lcal^n)$. Let $s_j=s_j^0+s_j^1$ be the orthogonal decomposition with respect to the splitting $\ker T_A\oplus (\ker T_A)^\perp$. In particular, there is a constant $c>0$ such that 
\begin{equation} \label{eqn:T-est}
\Vert T_A s_j^1\Vert\geq c\Vert s_j^1\Vert\ .
\end{equation} 
But then
$$
0=T_{A_j}s_j=(T_{A_j}-T_A)s_j + T_A s_j^1
$$
and so 
$$
(T_A-T_{A_j})s_j=T_A s_j^1 \ \Longrightarrow \ \Vert T_A s_j^1\Vert \to 0\ .
$$
The estimate \eqref{eqn:T-est} implies $s_j^1\to 0$. 
Hence, $s\in \ker T_A$, and continuity of $\sigma$ is proven.  The uniformity of $m$ in the second statement follows from Proposition \ref{prop:maruyama}.
\end{proof}

The semistable quotients in $\Quot(\Hcal,\tau_E)$ form an open set \cite[Thm.\ 2.8]{Maruyama:76}. Combining this with Theorem \ref{thm:classifying} we obtain

\begin{corollary} \label{cor:open}
The set $\Acal^{ss}(E,h)$ is open in $\Acal^{1,1}(E,h)$.
\end{corollary}

\section{Proof of the Main Theorem} \label{sec:proof}
As seen in Section \ref{sec:quot}, a consequence of the assumption that $X$ is  projective is a representation of holomorphic bundles and Uhlenbeck limits as quotients. The existence of many holomorphic sections also passes to certain line bundles on moduli spaces. This fact implies strong separation properties and will be used in this section to deduce the Main Theorem.

Let $A\in \Acal^{ss}(E,h)$ be a smooth integrable unitary connection such that the induced holomorphic bundle $\Ecal=(E,\dbar_{A})$ is semistable. 
Then there is a Seshadri filtration $\{0\}=\Fcal_0\subset\Fcal_1\subset\cdots \subset \Fcal_\ell=\Ecal$ such that the successive quotients $\Qcal_i=\Fcal_i/\Fcal_{i-1}$, $i=1,\ldots,\ell$ are stable torsion-free sheaves all of equal slope to that of $\Ecal_0$. 
Let $\Gr(\Ecal)=\oplus_{i=1}^\ell \Qcal_i$ and $\Ccal$ the cycle defined by the codimension $2$ support of $\Gr(\Ecal)^{\ast\ast}/\Gr(\Ecal)$ (see Section \ref{sec:cycle}). 
By the result of Bando-Siu referenced in Section \ref{sec:uhlenbeck}, there is an admissible HYM connection $A$ on $\Gr(\Ecal)^{\ast\ast}$, such that $(A,\Ccal, S(A))$
 defines an ideal HYM connection in the sense of Definition \ref{def:ideal-connection}. 


\begin{theorem}[\cite{DaskalWentworth:04,DaskalWentworth:07a,Sibley:15,SibleyWentworth:15}] \label{thm:flow}
Let $A_0\in\Acal^{ss}(E,h)$.  Then the Yang-Mills flow $A_t$ with initial condition $A_0$ converges in the sense of Theorem \ref{thm:ideal-convergence} to an ideal connection $[A_\infty, \Ccal_\infty, S(A_\infty)]$, where $A_\infty$ is the admissible HYM connection on $\Gr(\Ecal_0)^{\ast\ast}$, and $\Ccal_\infty$ is the codimension $2$ cycle defined by the torsion-free sheaf $\Gr(\Ecal_0)$. 
\end{theorem}
The theorem above states that the map $\Fscr$ in \eqref{eqn:F} is given purely in terms of the holomorphic initial data and the solution for admissible HYM connections on reflexive sheaves.

 We wish to prove that $\Fscr$ is continuous. For this we invoke the moduli space construction of Greb-Toma \cite{GrebToma:17}.
 Let $R^{\mu ss}\subset \Quot(\Hcal, \tau_E)$ denote the open subset consisting of quotients that are slope semistable torsion-free sheaves. Then there exists a (seminormal) projective variety $M^{\mu ss}$ and  a morphism (in particular, continuous map)
  $R^{\mu ss}\to M^{\mu ss} : \Fcal\mapsto [\Fcal]$ with the following properties:
\begin{enumerate}
	\item If $\Fcal_1\simeq \Fcal_2$, then $[\Fcal_1]=[\Fcal_2]$
  in $M^{\mu ss}$ (cf.\ the discussion preceding \cite[Def.\ 2.19]{GSTW:18}); 
  \item If $[\Fcal_1]=[\Fcal_2]$ in $M^{\mu ss}$, then $\Fcal_1^{\ast\ast}\simeq \Fcal_2^{\ast\ast}$ and $\Ccal_{\Fcal_1}=\Ccal_{\Fcal_2}$ \cite[Thm.\ 5.5]{GrebToma:17}.
\end{enumerate}
The association $[\Fcal]\to (\Fcal^{\ast\ast}, \Ccal_\Fcal)$ gives a well-defined map $\overline \Phi : \overline M^\mu(E,h)\to \overline M_{\HYM}(E,h)$, where
$\overline{M}^{\mu }(E,h)$ is the
closure of $M^\ast_{\HYM}(E,h)$ in  $M^{\mu ss}$.
   There is a diagram

 \begin{equation}
 \begin{split} \label{eqn:moduli}
\xymatrix{
\overline{\Acal^s}(E,h) \ar[r]^Q\ar[dr]_{\Fscr} & \overline M^\mu(E,h) \ar[d]^{\overline\Phi}\\	
&\overline M_{\HYM}(E,h)
}
 \end{split}
 \end{equation}
Here, the map $\overline{A^s}(E,h)\stackrel{Q}{\lra} \overline{M}%
^{\mu}(E,h)$ is defined by realizing a semistable bundle as a quotient in $R^{\mu ss}$ (see the discussion following Proposition \ref{prop:maruyama}), and sending this
quotient to its equivalence class in $M^{\mu ss}$. 

    By construction, $Q$ may be locally exhibited as the composition of the map $R^{\mu ss}\to M^{\mu ss}$ with a classifying map $\sigma$ as discussed in the previous section. The former map is a morphism of complex spaces and is therefore continuous. By Theorem \ref{thm:classifying}, $\sigma$ is  continuous as well. Since continuity is a local property, we deduce the continuity of $Q$.
Now one of the main results of \cite{GSTW:18} is Theorem 4.11, which states that
 $\overline \Phi$ is also continuous. 
By Theorem \ref{thm:flow}, the diagram \eqref{eqn:moduli} commutes, and we therefore conclude that $\mathscr F$ is
continuous on $\overline{A^s}(E,h)$.  

To address the general situation, we first reduce the problem as in \cite[Sec.\ 4]{DaskalWentworth:07a}. Let
 $A_i\to A_0$ be a sequence in $\Acal^{ss}(E,h)$ converging in the $C^\infty$ topology, and let 
$[A_\infty, \Ccal_\infty^A, S(A_\infty)]=\Fscr(A_0)$. By the compactness theorem \cite[Thm.\ 3.23]{GSTW:18}, we may assume that, after passing to a subsequence, there is an ideal connection  such that  $\Fscr(A_i)\to [(B_\infty,  \Ccal_\infty^B, S(B_\infty)]$. We must show that the two limits agree.

Let $A_{i,t}$ denote the Yang-Mills flow at time $t$ of $A_i$. Smooth dependence on initial conditions implies that for each fixed $T>0$, $A_{i,t}\to A_t$ smoothly as $i\to +\infty$, uniformly for $t\in [0,T)$.
\begin{lemma} \label{lem:A-subsequence}
 There is a subsequence $($also denoted {$\{i\}$}$)$ and 
$t_i\to +\infty$, such that $A_{i,t_i}\to [A_\infty, \Ccal_\infty^A, S(A_\infty)]$ in the sense of Theorem \ref{thm:ideal-convergence}. 
\end{lemma}

\begin{proof}
The proof relies on several properties.  First,  since $A_{i,t}\to A_t$ for every $t\geq 0$, by a diagonalization argument we may choose a  sequence $A_{i,t_i}$ so that (up to gauge), $A_{i,t_i}\to A_\infty$ weakly in $L^p_{1,loc}$ away from $|\Ccal|\cup S(A_\infty)$. 
Next, by the result in \cite{HongTian:04}, any sequence $A_{i,t_i}$, $t_i\to +\infty$, has a subsequence that converges to an ideal connection. This is shown in \cite{HongTian:04} for a sequence of times along a single flow, but the argument extends more generally. The key points are  Theorem 8 and Proposition 9 of \cite{HongTian:04}, and these hold uniformly for a smoothly convergent  sequence of initial conditions. Note that there is a uniform bound on the Hermitian-Einstein tensor. Given this fact, we are exactly in the set-up of the proof of \cite[Proposition 3.20]{GSTW:18}.  The conclusion of that result is that any limiting ideal HYM connection  of
$\{A_{i,t_i}\}$ must coincide with $[A_\infty, \Ccal_\infty^A, S(A_\infty)]$. 
\end{proof}

The Yang-Mills flow lies in a single complex gauge orbit, and $\Fscr(A_{i,t})=\Fscr(A_i)$. Therefore,
using  the same argument as above applied to connections along the flow, we also have the following.

\begin{lemma} \label{lem:B-subsequence}
There are complex gauge transformations $g_i$ such that if $B_i=g_i(A_{i,t_i})$, then after passing to a subsequence, 
$B_i\to [(B_\infty,  \Ccal^B_\infty, S(B_\infty)]$ in the sense of Theorem \ref{thm:ideal-convergence}.
\end{lemma}

We now apply Propositions \ref{prop:limit_sections} and \ref{prop:cycles} to both sequences $A_{i,t_i}$ and $B_i$. One obtains quotients
 $q_i^A\to q_\infty^A: \Hcal\to \Fcal^A_\infty$ and $q_i^B\to q_\infty^B: \Hcal\to \Fcal^B_\infty$ in 
$\Quot(\Hcal, \tau_E)$. Moreover, $(\Fcal_\infty^A)^{\ast\ast}\simeq \Ecal^A_\infty$ and $(\Fcal_\infty^B)^{\ast\ast}\simeq \Ecal^B_\infty$. In particular, 
since
$\Ecal_\infty^{A,B}$ have admissible Hermitian-Einstein metrics, the 
 $\Fcal_\infty^{A,B}$ are slope polystable, and so  they lie in $R^{\mu ss}$. 
Also, $\mathcal{C}^{A}_\infty=\mathcal{C}_{\mathcal{F}_{\infty }^{A}}$ and $\mathcal{C}%
_\infty^{B}=\mathcal{C}_{\mathcal{F}_{\infty }^{B}}$.

Now the quotients $q_i^A$ and $q_i^B$ are isomorphic for each $i$, since $B_i=g_i(A_{i,t_i})$. Furthermore,  these bundles are semistable. 
Hence, by item (1) above,  $[\Fcal_i^A]=[\Fcal_i^B]$ in $M^{\mu ss}$ for every $i$.
Since their limits 
are also semistable (in fact polystable), we conclude again from item (1) above and the continuity of the projection to $M^{\mu ss}$ that $[\mathcal{F}_{\infty }^{A}]=[\mathcal{F}_{\infty }^{B}]$. It then follows from item
 (2) that $\Ecal_\infty^A\simeq \Ecal_\infty^B$, and $\Ccal_\infty^A=\mathcal{C}_{\mathcal{F}_{\infty }^{A}}=\mathcal{C}_{\mathcal{F}_{\infty }^{B}}=\Ccal_\infty^B$.
 From the discussion following Definition \ref{def:ideal-connection}, the limiting ideal HYM connections coincide.
  This completes the proof of the Main Theorem.



\vspace{0.1cm}

\bibliographystyle{amsplain}
\bibliography{papers}{}

\vspace{0.6cm}

\begin{center}
---------------------------------------
\end{center}
\vspace{0.6cm}

\end{document}